\documentclass[a4paper,12pt]{amsart}
\usepackage[cp1251]{inputenc}
\usepackage{amssymb,upref}
\usepackage{graphicx}
\theoremstyle{plain}
\newtheorem{theorem}{Theorem}

\newtheorem{corollary}{Corollary}
\newtheorem{property}{Property}
\newtheorem{newproperty}{Property}

\theoremstyle{definition}
\newtheorem{definition}{Definition}
\theoremstyle{remark}

\uccode`ґ=`Ґ \uccode`Ґ=`Ґ \lccode`Ґ=`ґ \lccode`ґ=`ґ %
\uccode`є=`Є \uccode`Є=`Є \lccode`Є=`є \lccode`є=`є %
\uccode`і=`І \uccode`І=`І \lccode`І=`і \lccode`і=`і %
\uccode`ї=`Ї \uccode`Ї=`Ї \lccode`Ї=`ї \lccode`ї=`ї %

\begin{document}

\title[Fractal analysis of Guthrie-Nymann's set] {Fractal analysis of Guthrie-Nymann's set and its generalisations}
\author[M. Pratsiovytyi, D. Karvatskyi]{{M. Pratsiovytyi, D. Karvatskyi}
\centerline{(Institute of Mathematics of NAS of Ukraine)}}


\begin{abstract}
In this paper, we study the fractal properties of the boundary of the Cantorval connected with Guthrie-Nymann's series. In particular, we prove that such a Cantorval can be represented as a disjoint union of open intervals and a Cantor set having zero Lebesgue measure and a fractional Hausdorff dimension. Moreover, we extend the result to a countable family of Cantorvals with a similar structure.
\end{abstract}
\maketitle

\begin{section}{History of the study of Cantorvals}

In 1994, mathematicians P. Mendes and F. Oliveira studied the topological structure of the arithmetical sum  
$$C_1 \oplus C_2=\{ x_1 + x_2 : x_1 \in C_1, x_2 \in C_2\}$$
of two Cantor sets $C_1$ and $C_2$, both with Lebesgue measure zero \cite{Mendes} (see also \cite{Anisca}). As a product of such a sum, there occurs a perfect set, which is made up of points, intervals and gaps, each one being bothsided accumulated by all of the three. In simple words, this set is a strange combination of a nowhere dense set (Cantor set) and intervals. Because of its topological structure, the authors called this set \emph{Cantorval}. This set can be formally defined as a nonempty compact real subspace, which is the closure of its interior and the endpoints of any nontrivial component of this set are accumulation points of trivial components.

Analogously, Cantorval is one of the potential forms of the algebraic difference of a Cantor set $A$ defined by 
$$A \ominus A = \{ x_1 - x_2 : x_1, x_2 \in A \}$$
or the algebraic difference for two Cantor sets $A$ and $B$ in a more general case. For instance, some sufficient conditions for the set $A$ to have a Cantorval as the set of difference $A \ominus A$ were found in \cite{FN}.

Cantorvals are also one of the three possible topological types of the set of subsums for a convergent positive series $\sum a_n$, i.e., the set
\begin{equation*}
 \label{incomplit sum}
E(a_n)=\left\{\sum^{\infty }_{n=1}{{\varepsilon }_na_n}: ~ ~ ~ ({\varepsilon }_n) \in \{0, 1\}^{N} \right\}.
\end{equation*} 
It is well known that $E(a_n)$ for a convergent positive series is a symmetrical perfect set such that $E(a_n) \subseteq [0, r]$, where $r=\sum_{n=1}^{\infty} a_n$. In particular, the papers \cite{Guthrie 1988} and \cite{Nymann} give us the full topological classification of $E(a_n)$.
\begin{theorem}
\label{GN-Theorem}
The set of subsums $E(a_n)$ for a convergent positive series is one of the following three types: 
\begin{enumerate}
\item a finite union of closed intervals;
\item homeomorphic to the Cantor set (or shortly Cantor set);
\item M-Cantorval (or shortly Cantorval).
\end{enumerate}
\end{theorem}
It was also proven that every Cantorval is homeomorphic to the set $T^*$, defined by 

$$T^* \equiv C \cup \bigcup_{n=1}^{\infty} G_{2n-1} \equiv [0, 1] \setminus \bigcup_{n=1}^{\infty} G_{2n},$$
where $C$ is the Cantor ternary set, $G_n$ is the union of the $2^{n-1}$ open middle thirds, which are removed from $[0, 1]$ at the $n$-th step in the construction of $C$.

A large number of articles are devoted to the search for the necessary and sufficient conditions for the set of subsums to be a Cantorval. Despite essential progress for some series, the problem is quite difficult in general frameworks. The most significant results in that direction were obtained in \cite{Bartoszewicz} and \cite{Ferdinands} for multigeometric series  
$$k_1+k_2+\dots+k_m+k_1q+\dots+k_mq+\dots k_1q^i+\dots+k_mq^i+\dots,$$
($k_1, k_2, \dots, k_m$ are fixed positive scalars, $q \in (0, 1)$) and its generalisations \cite{KMV}, \cite{PKC}.
Some conditions for the set of subsums of a non-multigeometric series were found in \cite{Vinishin}. 


Finally, Cantorval can be an attractor for iterated function systems (see \cite{Banakh} and \cite{Banakiewicz}), which doesn't satisfy the open set condition as have significant overlaps. 
Thus, we draw the conclusion that Cantorvals naturally appear in various branches of mathematics, such as mathematical analysis, dynamical systems, number theory, and probability theory. In this context, the interest in this object is quite high.

In the next section, we shall study one of the most simple and, at the same time, beautiful and well-known examples of Cantorval.

\end{section}

\begin{section}{Metric and fractal properties of Guthrie-Nymann's set}

In \cite{Guthrie 1988}, John Guthrie and John Nymann considered the set 
$$X=\left\{x \mid x=\sum_{n=1}^{\infty} \frac{\alpha_n}{4^n}, \alpha_n \in \{0, 2, 3, 5 \} \right\},$$
of subsums for the series
\begin{equation}
\label{GNS}
\frac{3}{4}+\frac{2}{4}+\frac{3}{4^2}+\frac{2}{4^2}+\frac{3}{4^3}+\frac{2}{4^3}+\dots+\frac{3}{4^n}+\frac{2}{4^n}+\dots,
\end{equation}
which will further be called \emph{Guthrie-Nymann's series}. Such a set contains the interval $\left[ 2/3, 1 \right]$, but it is not a finite union of closed intervals; it combines the properties of a nowhere dense set and an infinite union of intervals. Some modifications of the set $X$ were also studied in \cite{PKS}. 


Following the nomenclature in \cite{Mendes} and \cite{Nymann}, we sketch some basic topological properties of $X$ (any arbitrary perfect set). Connectivity components of $X \subset R$ are either closed intervals or singletons. Intervals that are connectivity components of the closed set $X$ will be called \emph{X-intervals}, while one-point connectivity components of $X$ will be called \emph{loose points of} $X$.
Bounded open intervals that are connectivity components of the complement $X'=R \setminus X$ will be called \emph{X-gaps}. Next, we denote by $X_I$ the interior of the set $X$, and by $X_C$ we denote the boundary of $X$, hence $X=X_I \cup X_C$. One can observe that $X_I$ consists of open intervals that are interiors for $X$-intervals. Meanwhile, the boundary $X_C$ is the union of all loose points and endpoints of $X$-intervals. Corresponding between intervals and gaps for an arbitrary Cantorval was also described in \cite{Prus}. 

The paper \cite{Bielas} describes the topological and metric properties of the set $X$. In particular, it was proved that $X$ contains closed intervals having some special form and calculated the Lebesgue measure for the set $X_I$. To indicate all $X$-intervals, the authors used the notion about the center of distances. In particular, they exhibited the density of the set $K_n$, defined by
$$K_n=\left[\frac{2}{3}, 1\right] \bigcap \left\{ \sum_{i=1}^{n} \frac{x_i}{4^i} : \forall_{i} ~ x_i \in \{ 0, 2, 3, 5\} \right\},$$
in the interval $\left[ \frac{2}{3}, 1\right]$. The computation of the Lebesgue measure is based on the following important structural features of $X$:

\begin{property}
\label{P_1}
The interval $\left[\frac{2}{3}, 1\right]$ is included in the Cantorval $X$.
\end{property}

\begin{property}
\label{P_2}
The subset $X \setminus [\frac{2}{3}, 1] \subset X$ is the union of pairwise disjoint affine copies of $X$. In particular, this union includes two isometric copies of $C_n=\frac{1}{4^n} \cdot X$, for every $n>0$.
\end{property}

\begin{property}
\label{P_3}
The subset $X \setminus \left((\frac{2}{3}, 1) \bigcup (\frac{1}{6}, \frac{1}{4}) \bigcup (\frac{17}{12}, \frac{3}{2})\right) \subset X$ is the union of six pairwise disjoint affine copies $D=[0, \frac{1}{6}] \bigcap X$.
\end{property}



As a consequence of the above properties, we get the main corollary of the paper \cite{Bielas}.

\begin{corollary}
The Cantorval $X \subset \left[ 0, \frac{5}{3} \right]$ has Lebesgue measure 1.
\end{corollary}

According to Properties 1-3, there exists one-to-one correspondence between $X$-intervals and $X$-gaps. Moreover, for the set $X$ the length of all intervals and gaps equals the sum of Guthrie-Nymann's series. It follows that the Lebesgue measure of $X$ is concentrated only on its component intervals, namely on the set $X_I$, while $X_C$ has zero Lebesgue measure.

It is worth mentioning that a decent number of articles are dedicated to solving various topological, metric, and probabilistic problems connected with Guthrie-Nymann's set and other Cantorvals. For instance, the article \cite{Glab} is devoted to the study of the set of uniqueness for Cantorvals, i.e., the set $U(X)$ of all elements of $X$ such that it has only one representation by Guthrie-Nymann's series subsum \eqref{GNS}. It was proved that $U(X)$ is dense in $X$.


We recall the definition and some basic properties of the fractal dimension. Let $E$ be a bounded non-empty set in Euclidean space $R^n$. We define $|E|$ as \emph{diameter} for the set $E$ given by $|E|=\displaystyle \sup\{ |x-y| : x, y \in E \}.$

Let $\varepsilon$ be a fixed positive number. A finite or countable family of sets $\{ E_j \}$ is called $\varepsilon$-cover of the set $E$ if it satisfies the conditions $$E \subset \bigcup E_j, ~ \text{where} ~ |E_j| \leq \varepsilon, E_j \in R^n, \forall j \in N.$$
For any $\varepsilon > 0$ and non-negative $\alpha$, we write 
$$H^{\alpha}_{\varepsilon}(E) \equiv \inf_{|E_j| \leq \varepsilon} \displaystyle \{ |E_j|^{\alpha}: \{ E_j \} ~ \text{is a} ~ \varepsilon\text{-cover of} ~ E \}.$$

Next we define $\alpha$ - \emph{dimensional Hausdorff measure} $(H^{\alpha}-measure)$ of a bounded set $E$ as the value of the function
$$H^{\alpha}(E) \equiv \lim_{\varepsilon \rightarrow 0} H^{\alpha}_{\varepsilon}(E)=\sup_{\varepsilon > 0}H^{\alpha}_{\varepsilon}(E),$$
where we take the precise lower limit by all possible finite or countable coverings of the set $E$ with diameters less than $\varepsilon$. Finally, the critical value of $\alpha$ at which $H^{\alpha}(E)$ jumps from 0 to $\infty$, i.e., a non-negative number $\dim_{H}{E}$ such that
$$\dim_{H}{E}=\sup \{ \alpha : H^{\alpha}(E)=+\infty \}=\inf \{ \alpha : H^{\alpha}(E)=0 \}$$
is called \emph{Hausdorff} or \emph{fractal dimension} of the set $E$.

From now on, by $A \sqcup B$ we denote the disjoint union of the sets $A$ and $B$. Likewise, by $\bigsqcup_{k=1}^{n}A_k$ we denote the disjoint union of the sets $A_1, A_2,\dots , A_n$.

The involution $h: X \rightarrow X$ is defined by the formula
$$h \mapsto h[x]=\frac{5}{3}-x$$
is the symmetry transformation on $X$ with respect to the point 5/6.

To calculate Hausdorff dimension of the set $X_C$, we should first describe its structural properties.

\begin{property}
\label{p4}
The set $X_C$ is a union of pairwise disjoint affine copies of itself with similarity ratios of $1/4^n$, namely
$$X_C = \bigsqcup_{n \in N} \big(\bar{C}_{n}^{l} \sqcup \bar{C}_{n}^{r}\big),$$
where $\displaystyle \bar{C}^{l}_{n}=\sum_{i=1}^{n-1} \frac{2}{4^i} + \frac{1}{4^n} \cdot X_C$ and $\displaystyle \bar{C}^{r}_{n}=h\left[\sum_{i=1}^{n-1} \frac{2}{4^i} + \frac{1}{4^n} \cdot X_C \right]$ are right and left affine copy of $X_C$ with respect to the point $5/6$.
\end{property}

\begin{proof}

Due to Property \ref{P_2} (\cite{Bielas}, p. 693), $X \setminus \left[\frac{2}{3}, 1\right]$ consists of affine copies of itself having the form $C^l_1=\frac{1}{4} \cdot X$ and $C^r_1=h[C^l_1]$, $C^{l}_{2}=\frac{1}{2}+\frac{1}{4^2} \cdot X$ and $C^{r}_{2}=h[C^{l}_{2}]$, $\dots$, $C_{n}^{l}=\sum_{i=1}^{n-1}{\frac{2}{4^n}}+\frac{1}{4^{n}} \cdot X$ and $\displaystyle C_{n}^{r}=h\left[C_{n}^{l}\right]$, and so on. Since $X_C = X \setminus X_I$, every $X$-gap and every exterior of a $X$-interval is a $X_C$-gap. Thus, we have

$$ X \setminus \left[\frac{2}{3}, 1 \right] = \bigsqcup_{n=1}^{\infty} \left( C^{l}_{n} \sqcup C^{r}_{n} \right), $$
It follows that
$$X_C=X \setminus X_I = \left( \bigcup_{n=1}^{\infty} \left( C^{l}_{n} \cup C^{r}_{n}\right) \cup \left[ \frac{2}{3}, 1 \right] \right) \setminus X_I.$$

Since $\left(\frac{2}{3}, 1 \right) \subset X_I$, and defined $\bar{C}^{l}_{n}$ and $\bar{C}^{r}_{n}$ by 
$$\displaystyle \bar{C}^{l}_{n}={C}^{l}_{n} \setminus X_I = \sum_{i=1}^{n-1} \frac{2}{4^i} + \frac{1}{4^n} \cdot X_C,$$
$$\displaystyle \bar{C}^{r}_{n}=C^{r}_{n} \setminus X_I =h\left[\sum_{i=1}^{n-1} \frac{2}{4^i} + \frac{1}{4^n} \cdot X_C \right]=h\left[ \bar{C}^{l}_{n}\right],$$ 
we obtain that
$$X_C \equiv \bigsqcup_{n \in N} \big(\bar{C}_{n}^{l} \sqcup \bar{C}_{n}^{r}\big).$$
The endpoints of the central $X$-interval 2/3 and 1 are the limit points for the sequences $\left(\max C^{l}_{n}\right)_{n=1}^{\infty} \subset X_C$ and $\left(\min C^{r}_{l}\right)_{n=1}^{\infty} \subset X_C$, respectively. The boundary of a set is always closed. So, 2/3 and 1 are contained in $X_C$ as well as endpoints of an arbitrary $X$-interval.

\end{proof}

From Property \ref{p4}, we deduce that $X_C$ consists of countable disjoint affine copies of itself. The set having such properties is called \emph{$N$-self-similar set} (see \cite{PM} and \cite{PF}).
 
\begin{definition}

A bounded set $E \subset R^{n}$ is called \emph{N-self-similar} (or shortly \emph{NSS-set}) if there exists a countable collection of similarities $f_1, f_2, ..., f_n, ...$ such that
\begin{equation}
\label{NSS-condition}
\left\{
\begin{array}{l}
  1) E=f_1(E) \cup f_2(E) \cup f_3(E) \cup \dots \cup f_n(E) \cup \dots ; \\
  2) f_i(E) \bigcap f_j(E) = \emptyset ~ ~ ~ ~ ~ for ~ ~ ~ ~ ~ i \neq j.
\end{array}
\right.
\end{equation}
\end{definition}

NSS-set is different from self-similar sets because it consists of a countable number of self-copies instead of a finite number of copies.
In this context, NSS-set doesn't satisfy the standard open cover condition, and as a consequence, it is not so easy to calculate its Hausdorff dimension.
Let $k_1, k_2, \dots, k_n, \dots$ denote similarity factors for the correspondent contractors $f_1, f_2, \dots, f_n, \dots$, i.e. $k_i=d(f_i(E))/d(E)<1.$
\begin{definition}
Let $E$ satisfies the condition \eqref{NSS-condition}. The number $\alpha_{*}=\alpha_{*}(E)$, which is the solution of the following equation
$$k_1^{x}+k_2^{x}+\dots+k_{n}^{x}+\dots=1$$
is called the \emph{N-self-similar dimension} of the set $E$.
\end{definition}

We remark that there exist numerical sequences $(k_n)$ such that the equation above doesn't have a solution in the tradition sense. In that case, we shall consider that the number $a_{*}$ defined as
$$a_{*}=a_{*}(E)=\sup_{n} \{ x : k_1^{x}+k_2^{x}+\dots+k_{n}^{x}=1\}=\lim_{n \rightarrow \infty} \{ x_n : k_1^{x_n}+k_2^{x_n}+\dots+k_{n}^{x_n}=1\}$$
is the solution of the equation above.

\begin{property}
\label{P5}
The distance between neighbor affine copies $\bar{C}^{l}_{n}, \bar{C}^{l}_{n+1} \subset X_C$ and \newline $\bar{C}^{r}_{n}, \bar{C}^{r}_{n+1} \subset X_C$ can be calculated by the formula
$$d_n=d(\bar{C}^{l}_{n}, \bar{C}^{l}_{n+1})=d(\bar{C}^{r}_{n}, \bar{C}^{r}_{n+1})=\frac{1}{3} \cdot \frac{1}{4^n}.$$
\end{property}

\begin{proof}
Since $\displaystyle \bar{C}^{l}_{n}=\sum_{i=1}^{n-1} \frac{2}{4^i} + \frac{1}{4^n} \cdot X_C$ and $\displaystyle \bar{C}^{l}_{n+1}=\sum_{i=1}^{n} \frac{2}{4^i} + \frac{1}{4^{n+1}} \cdot X_C$ we get

$$\max \bar{C}^{l}_{n} = \sum_{i=1}^{n-1} \frac{2}{4^i} + \frac{1}{4^n} \cdot \max X_C = \sum_{i=1}^{n-1} \frac{2}{4^i} + \frac{1}{4^n} \cdot \frac{5}{3}.$$

On the other hand,

$$\min \bar{C}^{l}_{n+1}=\sum_{i=1}^{n} \frac{2}{4^i} + \frac{1}{4^{n+1}} \cdot \min X_C = \sum_{i=1}^{n} \frac{2}{4^i}.$$

Thus, we observe that 
$$d(\bar{C}^{l}_{n}, \bar{C}^{l}_{n+1}) = \min \bar{C}^{l}_{n+1} - \max \bar{C}^{l}_{n}=\frac{2}{4^{n+1}}-\frac{1}{4^n} \cdot \frac{5}{3}=\frac{1}{3} \cdot \frac{1}{4^n}.$$
Similarly, we can proof that $d(C^{r}_{n}, C^{r}_{n+1})=\frac{1}{3} \cdot \frac{1}{4^n}$. Here we see that $d_n \rightarrow 0$ when $n \rightarrow \infty$. 

\end{proof}

\begin{property}
\label{P6}
The distance between symmetrical with respect to the point $5/6$ affine copies $\bar{C}^{l}_{n}, \bar{C}^{r}_{n} \subset X_C$ can be calculated by the formula
$$s_n=d(\bar{C}^{l}_{n}, \bar{C}^{r}_{n})= \frac{5}{3} - 2 \cdot \left( \sum_{i=1}^{n} \frac{2}{4^i} + \frac{1}{4^n} \cdot \frac{5}{3}\right).$$
\end{property}

\begin{proof}
Taking into account the correspondence between $\bar{C}^{l}_{n}$ and $\bar{C}^{r}_{n}$, it is easy to observe that
$$d(\bar{C}^{l}_{n}, \bar{C}^{r}_{n})=\min \bar{C}^{r}_{n} - \max \bar{C}^{l}_{n}.$$
Since
$$\min \bar{C}^{r}_{n}=\frac{5}{3}-\left(\sum_{i=1}^{n-1} \frac{2}{4^i} + \frac{1}{4^n} \cdot \max X_C \right)=\frac{5}{3} - \sum_{i=1}^{n-1} \frac{2}{4^i} - \frac{1}{4^n} \cdot \frac{5}{3},$$
$$\max \bar{C}^{l}_{n}=\sum_{i=1}^{n-1} \frac{2}{4^i} + \frac{1}{4^n} \cdot \max X_C=\sum_{i=1}^{n-1} \frac{2}{4^i} + \frac{1}{4^n} \cdot \frac{5}{3},$$
we get
$$d(\bar{C}^{l}_{n}, \bar{C}^{r}_{n})=\frac{5}{3} - 2 \cdot \left( \sum_{i=1}^{n-1} \frac{2}{4^i} + \frac{1}{4^n} \cdot \frac{5}{3}\right).$$
Moreover, we observe that 
$$\lim_{n \rightarrow \infty}{\max \bar{C}^{l}_{n}}=\frac{2}{3} ~ ~ ~ ~ ~ \text{and} ~ ~ ~ ~ ~ \lim_{n \rightarrow \infty}{\min \bar{C}^{r}_{n}}=1,$$
that's why $s_n \rightarrow 1/3$ when $n \rightarrow \infty$.
\end{proof}

Due to Properties \ref{p4}-\ref{P6}, we can illustrate the set $X_C$ in the following picture:

\begin{figure}[h]
\center{\includegraphics[scale=0.32]{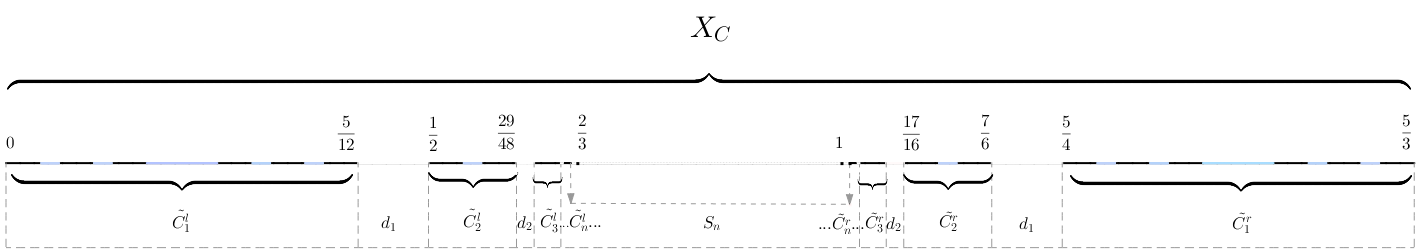}}
\caption{The set $X_C$ consists of countable disjoint affine copies $\bar{C}^{l}_{n}, \bar{C}^{r}_{n}$.}
\end{figure}

\begin{theorem}
The set $X_C$ is a NSS-set having the N-self-similar dimension equal to $a_{*}(X_C)=\log_{4} 3.$
\end{theorem}

\begin{proof}
Due to Property \ref{p4}, $X_C$ is a NSS-set as it consists of a countable union of pairwise disjoint affine copies of itself with similarity factors of $1/4^n, n \in N$. According to the definition, the NSS-dimension of $X_C$ can be found as the solution of the following equation:
\begin{equation}
\label{NSS-eq}
2\left( \frac{1}{4}\right)^x + 2\left( \frac{1}{4^2}\right)^x + 2\left( \frac{1}{4^3}\right)^x + \dots +2\left( \frac{1}{4^n}\right)^x + \dots= 1,
\end{equation}
from it follows
$$\frac{4^{-x}}{1-4^{-x}}=\frac{1}{2} ~ ~ ~ \Rightarrow ~ ~ ~ 4^{x}=3 ~ ~ ~ \Rightarrow ~ ~ ~ x=\log_{4} 3.$$
\end{proof}

We should note that the self-similar dimension (N-self-similar dimension) is not equal to the Hausdorff dimension in general. The simplest example of this fact is the set of all rational numbers $Q$ whose self-similar and N-self-similar dimensions equal 1, while its Hausdorff dimension equals 0.

\begin{theorem}
The Hausdorff dimension of the set $X_C$ is equal to its NSS-dimension.
\end{theorem} 

\begin{proof}
First things first, we observe that for $\alpha=\log_{4}{3}$ the identity \eqref{NSS-eq} holds. 
It means that $0<H^{\alpha}(X_C)<\infty$ when $\alpha=\log_{4}{3}$. Taking into account the property of the $\alpha$-dimensional Hausdorff measure, we get
$$H^{\alpha}(X_C)=H^{\alpha}\left( \bigsqcup_{n \in N} \big(\bar{C}_{n}^{l} \sqcup \bar{C}_{n}^{r} \big)\right)=\sum_{n=1}^{\infty}{H^{\alpha}(\bar{C}_{n}^{l})} + \sum_{n=1}^{\infty}{H^{\alpha}(\bar{C}_{n}^{r})}=$$
$$=H^{\alpha}(X_C)\left( 2 \cdot \left( \frac{1}{4}\right)^{\alpha} + 2 \cdot \left( \frac{1}{4^2}\right)^{\alpha} + \left( \frac{1}{4^3}\right)^{\alpha} + \dots +2 \cdot \left( \frac{1}{4^n}\right)^{\alpha} + \dots\right).$$
As a consequence, $\alpha=\log_{4} 3$ is the common value for the NSS-dimension and Hausdorff dimension of $X_C$. 
\end{proof}

\begin{corollary}
The Cantorval $X$, which is the set of subsums for Guthrie-Nymann's series, can be represented as
$X=X_{I} \bigsqcup X_{C}$, where $X_{I}$ is an infinity union of open intervals having the Lebesgue measure equal to 1, $X_{C}$ is a Cantor set with zero Lebesgue measure and fractional Hausdorff dimension $\dim_{H}X_{C}=\log_4 3$.
\end{corollary}
\end{section}

\begin{section}{Generalization of the main result}

In this section, we shall extend the above result to a countable family of Cantorvals having the following form:
$$X(m)=\left\{ \sum_{n=1}^{\infty} \frac{\varepsilon_n}{(2m+2)^n} : (\varepsilon_n) \in \{0, 2, 3, \dots, 2m+1, 2m+3 \}^N \right\},$$
which are sets of subsums for the series
$$3q+\underbrace{2q+\dots+2q}_m+3q^2+\underbrace{2q^2+\dots+2q^2}_m+\dots+3q^n+\underbrace{2q^n+\dots+2q^n}_m+\dots$$ 
where $q=1/(2m+2), m \in N$. It is worth mentioning that for $m=1$ we have case $X(1)=X$, which was well described in the previous section.

One can see \cite{Banakiewicz} that $X(m)$ is the attractor for an iterated function system $\{ w_i: R \rightarrow R \}$, where $i=1,\dots,2m+2$, $m \in N$, and all contractors $w_i$ are similarities given by
\begin{enumerate}
\item $\displaystyle w_1(x)=\frac{x}{2m+2}$;
\item $\displaystyle w_i(x)=\frac{i}{2m+2}+\frac{x}{2m+2}$ for $i \in \{2, 3, \dots, 2m+1 \}$;
\item $\displaystyle w_{2m+2}(x)=\frac{2m+3}{2m+2}+\frac{x}{2m+2}$.
\end{enumerate}
That means $$X(m)=\bigcup_{i=1}^{m+2}w_i(X(m))=W(X(m)).$$
\begin{theorem}(\cite{Banakiewicz})
\label{BT}
Let $m \in N$ and $I=\left[ 0, \frac{2m+3}{2m+1} \right]$. Then
$$X(m)= \bigcap_{n \in N}W^{n}(I),$$
where $W^n(I)=W(W^{n-1}(I))$ and $W^0=id$. Moreover,
$$W^{n}(I)=\bigsqcup_{k=0}^{n-1}w_2^k \circ w_1(W^{n-1-k}(I)) \sqcup \left[ w_2^n(0), w^{n}_{2m+1}\left(\frac{2m+3}{2m+1}\right)\right] \sqcup $$
$$\sqcup \bigsqcup_{k=0}^{n-1}w_{2m+1}^{n-1-k} \circ w_{2m+2}(W^k(I)).$$
\end{theorem}

\begin{theorem}
The Lebesgue measure of the M-Cantorval $X(m)$ is equal to 1 and it is equal to the sum of lengths of all its component intervals ($X(m)-intervals$). 
\end{theorem}

Analogously to the previous section, we define $X_{I}(m)$ as the interior, and by $X_{C}(m)$ we denote the boundary of the set $X(m)$. Moreover, $X(m)=X_I(m) \bigsqcup X_{C}(m)$.   

\begin{newproperty}
 \label{P1m}
The interval $\left[\frac{2}{2m+1}, 1\right]$ is included in the Cantorval $X(m)$.
\end{newproperty}
\begin{proof}
According to Theorem \ref{BT}, the intersection of nested intervals $\bigcap_{n=1}^{\infty} \left[ w_2^n(0), w^{n}_{2m+1}\left(\frac{2m+3}{2m+1}\right)\right]$ is included in $X(m)$. It is easy to see that
$$\lim_{n \rightarrow \infty}{w_2^n(0)}=\frac{2}{2m+2}+\frac{2}{(2m+2)^2}+\dots+\frac{2}{(2m+2)^n} +\dots = \frac{2}{2m+1}$$
while
$$\lim_{n \rightarrow \infty}w^{n}_{2m+1}\left(\frac{2m+3}{2m+1}\right)=\lim_{n \rightarrow \infty}{\frac{2m+3}{(2m+2)^n(2m+1)}+}$$
$$+\left(\frac{2m+1}{2m+2}+\frac{2m+1}{(2m+2)^2}+\dots+\frac{2m+1}{(2m+2)^n}+\dots\right) =1.$$
As a consequence, we obtain $\left[\frac{2}{2m+1}, 1\right] \subset X(m)$.
\end{proof}

For the collection of contractors $w_i$ holds the following basic statements:
\begin{enumerate}
\item $\min w_i(X(m))<\min w_{i+1}(X(m)),  1 \leq i \leq 2m+1;$
\item $\max w_i(X(m))<\max w_{i+1}(X(m)), 1 \leq i \leq 2m+1;$
\item $\min w_i(X(m))=\frac{i}{2m+2}, \ \ \  \max w_i(X(m))=\frac{i(2m+1)+2m+3}{(2m+1)(2m+2)}, \ \ \  2 \leq i \leq 2m+1;$
\item $\min w_1(X(m))=0, \ \ \  \max w_1(X(m))=\frac{2m+3}{(2m+1)(2m+2)};$
\item $\min w_{2m+2}(X(m))=\frac{2m+3}{2m+2}, \ \ \  \max w_{2m+2}(X(m))=\frac{2m+3}{2m+1}.$
\end{enumerate}

The involution $h^{(m)}: X(m) \rightarrow X(m)$ is defined by the formula
$$h^{(m)} \mapsto h^{(m)}[x]=\frac{2m+3}{2m+1}-x$$
is the symmetry transformation on $X(m)$ with respect to the point $(2m+3)/(4m+2)$.
\begin{newproperty}
\label{P2m}
The subset $X(m) \setminus \left[\frac{2}{2m+1}, 1\right] \subset X(m)$ is a union of pairwise disjoint affine copies of $X(m)$. In particular, this union includes two isometric copies $\frac{1}{(2m+2)^n} \cdot X(m)$, for every $n>0$.
\end{newproperty}
\begin{proof}
As $X(m)$ is the attractor for the IFS $\{ w_i: i=1,2, \dots, 2m+2\}$ we get
$$X(m)=w_1(X(m)) \cup w_2(X(m)) \cup w_3(X(m)) \cup \dots \cup w_{2m+1}(X(m)) \cup w_{2m+2}(X(m)).$$

We observe that
$$w_1(X(m)) \cap w_i(X(m)) = \emptyset, \ \text{for} \ 2 \leq i \leq 2m+2;$$
$$w_i(X(m)) \cap w_{2m+2}(X(m)) = \emptyset, \ \text{for} \ 1 \leq i \leq 2m+1;$$
$$w_1(X(m)) \cap \left[\frac{2}{2m+1}, 1\right] = \emptyset, \ \ \ w_{2m+2}(X(m)) \cap \left[\frac{2}{2m+1}, 1\right] = \emptyset.$$
It follows that $C_1^{l}(m)=w_1(X(m))=\frac{1}{2m+2} \cdot X(m)$ and $C_1^{r}(m)=w_{2m+2}(X(m))=h^{(m)}\left[C_{1}^{l}(m)\right]$ are two disjoint affine copies contained in $X(m) \setminus \left[\frac{2}{2m+1}, 1\right]$. It is easy to see that the sets $w_i(X(m))$ for $3 \leq i \leq 2m$ are totally contained in $\left[\frac{2}{2m+1}, 1\right]$. On the other hand, 
$$w_{i}(X(m)) \cap \left[\frac{2}{2m+1}, 1\right] \neq \emptyset, \ \text{but} \ w_{i}(X(m)) \not \subset \left[\frac{2}{2m+1}, 1\right], \ \text{for} \ i \in \{2, 2m+1\}.$$

For the sets $w_2 (X(m))$ and $w_{2m+1} (X(m))$ we have
$$w_{2} (X(m))=w_{2}\left(\bigcup_{i=1}^{2m+2} w_i(X(m))\right), w_{2m+1} (X(m))=w_{2m+1}\left(\bigcup_{i=1}^{2m+2} w_i(X(m))\right).$$
Likewise, we observe that
$$w_2 \circ w_1(X(m)) \cap w_2 \circ w_i(X(m)) = \emptyset, \ \text{for} \ 2 \leq i \leq 2m+2;$$
$$ w_{2m+1} \circ w_i(X(m)) \cap w_{2m+1} \circ w_{2m+2}(X(m)) = \emptyset, \ \text{for} \ 1 \leq i \leq 2m+1;$$
$$w_2 \circ w_1(X(m)) \cap \left[\frac{2}{2m+1}, 1\right] = \emptyset, w_{2m+1} \circ w_{2m+2}(X(m)) \cap \left[\frac{2}{2m+1}, 1\right] = \emptyset.$$
It follows that $C_2^{l}(m)=w_2 \circ w_1 (X(m))=\frac{2}{(2m+2)} + \frac{1}{(2m+2)^2} \cdot X(m)$ and $C_2^{r}(m)=w_{2m+1} \circ w_{2m+2}(X(m))=h^{(m)}\left[C_{2}^{l}(m)\right]$ are two disjoint affine copies contained in $X(m) \setminus \left[\frac{2}{2m+1}, 1\right]$. It is easy to see that the sets $w_2 \circ w_i(X(m))$ for $3 \leq i \leq 2m+2$ and $w_{2m+1} \circ w_j(X(m))$ for $1 \leq j \leq 2m$ are totally contained in $\left[\frac{2}{2m+1}, 1\right]$. On the other hand, 

$$w_{i} \circ w_{i}(X(m)) \cap \left[\frac{2}{2m+1}, 1\right] \neq \emptyset, \ \text{but} \ w_{i} \circ w_{i}(X(m))  \not \subset  \left[\frac{2}{2m+1}, 1\right], \ \text{for} \ i \in \{2, 2m+1\}.$$
By continuing the above considerations, we can show that $C^{l}_{n}(m), C^{r}_{n}(m)$ defined by
$$C_n^{l}(m)=\underbrace{w_2 \circ \dots \circ w_2}_{n-1} \circ w_1(X(m))=\sum_{i=1}^{n-1}{\frac{2}{(2m+2)^i}} + \frac{1}{(2m+2)^n} \cdot X(m),$$
$$C_n^{r}(m)=\underbrace{w_{2m+1} \circ \dots \circ w_{2m+1}}_{n-1} \circ w_{2m+2}(X(m))=h^{(m)}\left[C_{n}^{l}(m)\right],$$
are absolutely contained in $X(m) \setminus \left[\frac{2}{2m+1}, 1\right]$, for any $n \in N$, while almost all other iterations belong to the interval $\left[\frac{2}{2m+1}, 1\right]$. The conclusion is as follows: 
$$X(m) \setminus \left[\frac{2}{2m+1}, 1\right] = \bigsqcup_{n=1}^{\infty} \left( C^{l}_{n}(m) \sqcup C^{r}_{n}(m)\right).$$
\end{proof}


\begin{newproperty}
\label{P3m}
The set $X_C(m)$ is a union of pairwise disjoint affine copies of itself with similarity ratios of $1/(2m+2)^n$, $n \in N$, namely
$$X_C(m) = \bigsqcup_{n \in N} \big(\bar{C}_{n}^{l}(m) \sqcup \bar{C}_{n}^{r}(m)\big),$$
where $\displaystyle \bar{C}^{l}_{n}(m)=\sum_{i=1}^{n-1} \frac{2}{{(2m+2)}^i} + \frac{1}{{(2m+2)}^n} \cdot X_C(m)$ and $\displaystyle \bar{C}^{r}_{n}(m)=h^{m}\left[ C_{n}^{l}(m) \right]$ are right and left affine copies with respect to the point $(2m+3)/(4m+2)$.
\end{newproperty}

\begin{newproperty}
The distance between neighbour affine copies $\bar{C}^{l}_{n}(m), \bar{C}^{r}_{n+1}(m) \subset X_C(m)$ and $\bar{C}^{r}_{n}(m), \bar{C}^{r}_{n+1}(m) \subset X_C(m)$ can be calculated by the formula
$$d_n(m)=d(\bar{C}^{l}_{n}(m), \bar{C}^{l}_{n+1}(m))=d(\bar{C}^{r}_{n}(m), \bar{C}^{r}_{n+1}(m))=\frac{1}{2m+1} \cdot \frac{1}{(2m+2)^n}.$$
\end{newproperty}

\begin{newproperty}
\label{lastpm}
The distance between symmetrical affine copies $\bar{C}^{l}_{n}(m), \bar{C}^{r}_{n}(m) \subset X_C(m)$ can be calculated by the formula
$$s_n(m)=d(\bar{C}^{l}_{n}(m), \bar{C}^{r}_{n}(m))= \frac{2m+3}{2m+1} - 2 \cdot \left( \sum_{i=1}^{n-1} \frac{2}{(2m+2)^i} + \frac{1}{(2m+2)^n} \cdot \frac{2m+3}{2m+1}\right).$$
\end{newproperty}

Properties \ref{P3m}-\ref{lastpm} are direct corollaries of Property \ref{P2m}. Finally, we conclude that $X(m)$ has in some sense a similar structure to $X$, which might be illustrated in the following picture:

\begin{figure}[h]
\center{\includegraphics[scale=0.32]{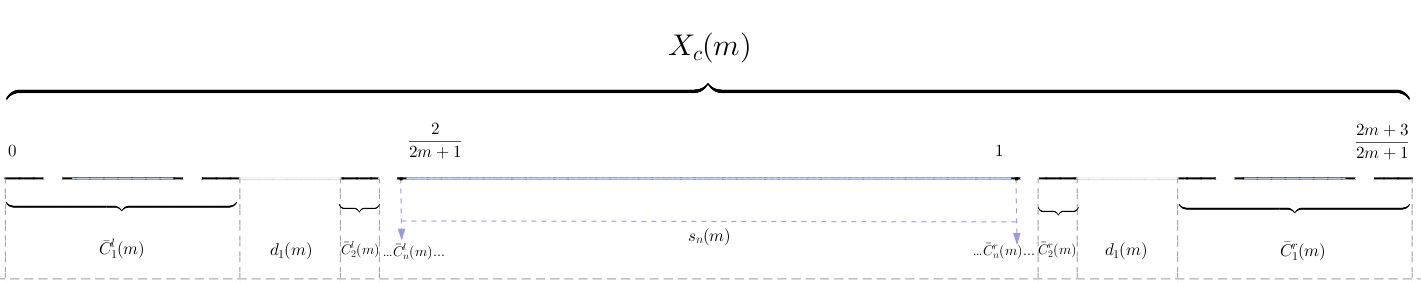}}
\caption{The set $X_C(m)$ consists of countable disjoint affine copies $\bar{C}^{l}_{n}(m), \bar{C}^{r}_{n}(m)$.}
\end{figure}

\begin{theorem}
The set $X_C(m)$ is a fractal set with the Hausdorff dimension equal to $\dim_{H}X_C(m)=\log_{2m+2}{3}$.
\end{theorem}
\begin{proof}
Due to Property \ref{P3m}, $X_C(m)$ is a NSS-set as it consists of a countable union of pairwise disjoint affine copies of itself with similarity factors of $1/(2m+2)^n, n \in N$. The NSS-dimension of $X_C$ can be found as the solution of the following equation:
\begin{equation*}
2\left( \frac{1}{2m+2}\right)^x + 2\left( \frac{1}{2m+2}\right)^{2x} + 2\left( \frac{1}{2m+2}\right)^{3x} + \dots +2\left( \frac{1}{2m+2}\right)^{nx} + \dots= 1,
\end{equation*}
from it follows
$$\frac{(2m+2)^{-x}}{1-(2m+2)^{-x}}=\frac{1}{2} ~ ~ ~ \Rightarrow ~ ~ ~ (2m+2)^{x}=3 ~ ~ ~ \Rightarrow ~ ~ ~ x=\log_{2m+2} 3.$$
Moreover, we observe that $0<H^{\alpha}(X_C(m))<\infty$ when $\alpha=\log_{2m+2}{3}$. Taking into account the property of the $\alpha$-dimensional Hausdorff measure, we get the identity $\alpha_{\ast}(X_C(m))=\dim_{H}X_{C}(m)=\log_{2m+2}{3}.$
\end{proof}

\begin{corollary}

The Cantorval $X(m)$ can be represented as the union
$X=X_{I}(m) \bigsqcup X_{C}(m)$, where $X_{I}(m)$ is an infinity union of open intervals having the Lebesgue measure equal to 1, $X_{C}(m)$ is a Cantor set with zero Lebesgue measure and fractional Hausdorff dimension $\dim_{H}X_{C}=\log_{2m+2} 3$.

\end{corollary}

\end{section}

\begin{section}{Open problems}

As a result of the paper, we demonstrated that the Cantorval $X(m)$ can be obtained as the disjoint union of open intervals $X_{I}(m)$ and the fractal set $X_{C}(m)$ with Hausdorff dimension $\dim_H X_{C}(m)=\log_{2m+2}{3}$. It is intriguing that the choice of $m$ doesn't affect the Lebesgue measure of the entire Cantorval while it impacts the Hausdorff dimension of its boundary. Moreover, we observe that $\dim_{H}X_{C}(m) \rightarrow 0$ when $m \rightarrow \infty$.

We say that a Cantorval $T$ is \emph{achievable} if there exists a convergent positive series $\sum u_n$ such that $T=E(u_n)$. Taking into account the above theorem, there naturally arise the following questions:
\begin{itemize}
\item Is it possible to represent an arbitrary achievable Cantorval $T$ as a union of open intervals $T_I$ and some Cantor set $T_C$ with zero Lebesgue measure? In other words, is the Lebesgue measure for a random achievable Cantorval concentrated only on its component intervals, as in the case of $X$?
\item Is it possible for the set $T_{C}$ to be a superfractal set (zero Lebesgue measure set with $\dim_{H} T_C = 1$) or an anomalyfractal set ($\dim_{H} T_C = 0$)?
\item Can we introduce a general technique for the calculation of the Hausdorff dimension for the boundary of an arbitrary achievable Cantorval?
\end{itemize}
\end{section}

Finally, we should note that potentially we can construct a set $T=T_I \cup T_C$ homeomorphic to the set $T^*$ (see Section 1) having a positive Lebesgue measure boundary $T_C$. However, there are no guarantees that $T$ is achievable. This brings up one more interesting question about the necessary and sufficient conditions for a set to be achievable.

\textbf{Acknowledgement.} The first author was supported by a grant from the Simon Foundation. The second author is working at the University of St. Andrews in the framework of the Isaac Newton Institute solidarity programme and has additional support from the London Mathematical Society.

\end{document}